\documentclass[a4paper]{amsart}
\usepackage{fixltx2e}
\usepackage[utf8]{inputenc}
\usepackage[T1]{fontenc}
\usepackage[charter]{mathdesign}
\usepackage[bookmarks=true,pdfdisplaydoctitle=true]{hyperref}
\hypersetup{
pdfpagemode=UseNone,
pdfstartview=FitH,
pdftitle={Cancellation for the simplex Hilbert transform},
pdfauthor={Pavel Zorin-Kranich},
pdflang=en-US
}

\usepackage[safeinputenc=true,style=alphabetic,firstinits,maxnames=4,doi=false,eprint=true,isbn=false,url=false]{biblatex}

\newbibmacro{string+doiurl}[1]{%
  \iffieldundef{doi}{%
    \iffieldundef{url}{%
         #1%
      }{%
      \href{\thefield{url}}{#1}%
    }%
  }{%
    \href{http://dx.doi.org/\thefield{doi}}{#1}%
  }%
}
\DeclareFieldFormat[article,misc]{title}{\usebibmacro{string+doiurl}{\mkbibquote{#1}}}

\newcommand*{\mailto}[1]{\href{mailto:#1}{#1}}

\numberwithin{equation}{section}
\newtheorem{theorem}[equation]{Theorem}
\newtheorem{lemma}[equation]{Lemma}
\newtheorem{proposition}[equation]{Proposition}
\newtheorem{corollary}[equation]{Corollary}
\theoremstyle{definition}

\theoremstyle{remark}

\newtheorem*{ack}{Acknowledgments}

\newcommand*{\N}{\mathbb{N}}
\newcommand*{\Z}{\mathbb{Z}}
\newcommand*{\R}{\mathbb{R}}
\newcommand*{\C}{\mathbb{C}}
\newcommand*{\dif}{\mathrm{d}}
\def\<{\left\langle}
\def\>{\right\rangle}

\newcommand{\DI}{\mathcal{I}}
\newcommand{\scale}{k}
\newcommand{\Scales}{\mathbf{S}}
\newcommand{\dm}{n} 

\begin{document}
\title{Cancellation for the simplex Hilbert transform}
\author{Pavel Zorin-Kranich}
\address{Universit\"at Bonn\\
  Mathematisches Institut\\
  Endenicher Allee 60\\
  53115 Bonn\\
  Germany
}
\email{\mailto{pzorin@uni-bonn.de}}
\urladdr{\url{http://www.math.uni-bonn.de/people/pzorin/}}
\begin{abstract}
We show that the truncated simplex Hilbert transform enjoys some cancellation in the sense that its norm grows sublinearly in the number of scales retained in the truncation.
This extends the recent result by Tao on cancellation for the multilinear Hilbert transform.
Our main tool is the Hilbert space regularity lemma due to Gowers, which enables a very short proof.
\end{abstract}
\maketitle
\allowdisplaybreaks[3]

\section{Introduction}
Consider the $(\dm+1)$-linear form defined on functions of $\dm$ variables
\begin{equation}
\label{eq:HHT-def-hom}
\Lambda_{K}(F_{0},\dots,F_{\dm})
:=
\int_{\R^{\dm+1}} \prod_{i=0}^{\dm}F_{i}(x_{(i)}) K(\sum_{i=0}^{\dm} x_{i}) \dif x,
\end{equation}
where $x_{(i)} = (x_{0},\dots,x_{i-1},x_{i+1},\dots,x_{\dm})$ denotes the omission of the $i$-th coordinate.
We are interested in $K$ being a smooth truncation of a one-dimensional Calder\'on--Zygmund kernel.
Recall that a one-dimensional Calder\'on--Zygmund kernel is a distribution $K$ that satisfies $\|\hat K\|_{\infty}\leq 1$ and that coincides, away from $0$, with a differentiable function $K$ such that $|K(t)|\leq|t|^{-1}$ and $|K'(t)|\leq |t|^{-2}$.
Our truncations have the form
\[
\psi_{\Scales} = \sum_{\scale\in \Scales} \psi_{\scale},
\quad\text{where}\quad
\psi_{\scale}(t)=\phi(2^{-\scale}t)K(t),
\]
$\Scales\subset\Z$ is an interval, and $\phi$ is an even, smooth function supported on $\pm[1,4]$ such that $\sum_{\scale\in\Z}\phi(2^{-\scale}t)=1$ for all $t\neq 0$.
We call the form $\Lambda_{\Scales} := \Lambda_{\psi_{\Scales}}$ the \emph{truncated simplex Hilbert transform}, in analogy to the truncated triangular Hilbert transform, to which this form specializes for $\dm=2$, $K(t)=\frac1t$.
The eponymous simplex is spanned by the set $\{0,\dots,\dm\}$.
Each function $F_{i}$ is associated to a side of the simplex and accepts the variables whose indices span that side.

Since $\|\psi_{\scale}\|_{1}=O(1)$ and by H\"older's inequality the estimate
\begin{equation}
\label{eq:trivial-estimate}
|\Lambda_{\Scales}(F_{0},\dots,F_{\dm})|
\lesssim_{\dm} |\Scales| \prod_{i=0}^{\dm} \| F_{i} \|_{p_{i}}
\end{equation}
is immediate for any Hölder tuple of exponents $1\leq p_{i}\leq\infty$.
Our main result, extending a recent result by Tao \cite{arxiv:1505.06479}, is the following qualitative improvement over this bound.
\begin{theorem}
\label{thm:tiny-gain}
Let $\dm\geq 1$.
Then for any $1<p_{i}<\infty$ with $\sum_{i=0}^{\dm} p_{i}^{-1}=1$ we have
\[
|\Lambda_{\Scales}(F_{0},\dots,F_{\dm})|
\leq
o_{\dm,p_{0},\dots,p_{\dm}}(|\Scales|) \prod_{i=0}^{\dm} \| F_{i} \|_{p_{i}}.
\]
\end{theorem}
The restriction to Hölder tuples of exponents is necessary; more in general, if the kernel $K$ is homogeneous of degree $d$, then the form \eqref{eq:HHT-def-hom} can only be bounded on $L^{p_{0}} \times \dots \times L^{p_{\dm}}$ if $\sum_{i=0}^{\dm}p_{i}^{-1}=1+(1+d)/\dm$.
Uniform bounds in $\Scales$ remain out of reach of current methods unless $\dm=1$, in which case $\Lambda_{\Scales}$ reduces to (the dual of) a usual truncation of $K$.

An appropriate choice of the functions $F_{i}$ (see Appendix~\ref{sec:max-mod}) shows that Theorem~\ref{thm:tiny-gain} yields some cancellation for the maximally modulated by polynomials of degree $d$, linearized, $(\dm-d)$-linear, truncated Calder\'on--Zygmund operator
\begin{equation}
\label{eq:max-mod-MHT}
C_{N_1,\ldots,N_{d}}(f_{d+1},\dots,f_{\dm})(x)
=
\mathrm{p.v.}\int_{\mathbb{R}} e^{i\sum_{j=1}^{d} N_j(x)t^{j}} \prod_{j=d+1}^{\dm} f_{j}(x-b_{j}t) \psi_{\Scales}(t) \dif t,
\end{equation}
where $b_{d+1},\dots,b_{\dm}$ are distinct non-zero real numbers and $N_{1},\dots,N_{d}:\R\to\R$ are measurable linearizing functions, with constants uniform in the choice of $b_{j}$'s and $N_{j}$'s.
This answers a question from \cite{arxiv:1505.06479}, where such cancellation was obtained for $d=0$ and commensurable $b_{j}$'s.

Our main tool is the Hilbert space regularity lemma due to Gowers \cite{MR2669681}, whose proof is so short that we chose to write it out in full in Section~\ref{sec:regularity}.
The multidimensional setup allows us to avoid the much harder arithmetic regularity lemma used in \cite{arxiv:1505.06479}.

Theorem~\ref{thm:tiny-gain} is proved by induction on $\dm$.
The case $\dm=1$ follows from the standard theory of truncated Calder\'on--Zygmund operators, see e.g.\ \cite[\textsection I.7]{MR1232192}.
In the inductive step we assume that the theorem holds with $\dm>1$ replaced by $\dm-1$.
Multilinear interpolation with the trivial estimate \eqref{eq:trivial-estimate} shows that it suffices to consider $p_{0},\dots,p_{\dm}>\dm$ and indicator functions $F_{i}=1_{E_{i}}$.
We make these assumptions throughout Section~\ref{sec:tree}, which contains a single tree estimate, and Section~\ref{sec:selection}, which describes a standard tree selection algorithm.

\section{The regularity lemma}
\label{sec:regularity}
The material in this section is almost identical to Gowers's original exposition in \cite{MR2669681}.
The only difference from the finite-dimensional case is that it turns out convenient to work with \emph{extended seminorms}, that is, functions $\|\cdot\|$ on a vector space $H$ taking values in the extended positive reals $[0,+\infty]$ that are subadditive, homogeneous, and map $0$ to $0$ (this observation has peen previously used to further streamline \cite{arxiv:1111.7292} Walsh's proof of the multilinear mean ergodic theorem \cite{MR2912715}).
The reason is that the atomic seminorms $\|\cdot\|_{\Sigma}$, defined below, are typically extended.
\begin{lemma}
\label{lem:Sigma-ext-seminorm}
Let $H$ be a Hilbert space and $\Sigma\subset H$.
Then the formula
\[
\|f\|_{\Sigma} := \inf\Big\{ \sum_{t}|\lambda_{t}| : f=\sum_{t}\lambda_{t}\sigma_{t}, \sigma_{t} \in \Sigma \Big\},
\]
where sums are finite (possibly empty), and the infimum of an empty set is by convention $+\infty$, defines an extended seminorm on $H$ whose dual extended seminorm is given by
\[
\|f\|_{\Sigma}^{*} := \sup_{\phi\in H: \|\phi\|_{\Sigma}\leq 1} |\<f,\phi\>| = \sup_{\sigma\in\Sigma}|\<f,\sigma\>|.
\]
\end{lemma}

Gowers's Hilbert space regularity lemma reads as follows.
\begin{theorem}
\label{thm:structure}
Let $\delta>0$ and $\eta \colon\R_+\to\R_+$ be any function.
Let $H$ be a Hilbert space with norm $\|\cdot\|_{H}$ and let $\|\cdot\|$ be an arbitrary further extended seminorm on $H$.
Then for every $f\in H$ with $\|f\|_{H} \leq 1$ there exists $C=O_{\delta,\eta}(1)$ and a decomposition
\begin{equation}
\label{eq:decomposition}
f = \sigma + u + v
\end{equation}
such that
\begin{equation}
\|\sigma\| < C, \quad
\|u\|^* < \eta(C), \quad\text{and}\quad
\|v\|_{H} < \delta.
\end{equation}
\end{theorem}
The proof uses the following separation lemma.
\begin{lemma}
\label{lem:sep}
Let $V_{i}$, $i=1,\dots,k$, be convex subsets of a Hilbert space $H$, at least one of which is open, and each of which contains $0$.
Let $V:=c_{1}V_{1}+\dots+c_{k}V_{k}$ with $c_{i}>0$ and take $f\not\in V$.
Then there exists a vector $\phi\in H$ such that $\<f,\phi\> \geq 1$ and $\Re\<v,\phi\> < c_{i}^{-1}$ for every $v\in V_{i}$ and every $i$.
\end{lemma}
\begin{proof}
By the assumption the set $V$ is open, convex and does not contain $f$.
By the Hahn--Banach theorem there exists a $\phi\in H$ such that $\<f,\phi\> \geq 1$ and $\Re\<v,\phi\> < 1$ for every $v\in V$.
The claim follows.
\end{proof}
There is also a constructive version of Lemma~\ref{lem:sep} with an $\epsilon$ loss, in the sense that the conclusion changes to $\Re\<v,\phi\> < (1+\epsilon)c_{i}^{-1}$ (this version still suffices for our purpose).
Indeed, since $V\ni 0$ is open and $f\not\in V$, we have $f\not\in(1+\epsilon)^{-1}\overline{V}$.
Let $g\in(1+\epsilon)^{-1}\overline{V}$ be the element that minimizes the distance from $f$ (such $g$ is unique).
One can then take $\phi = (f-g)/\<f-g,f\>$.

\begin{proof}[Proof of Theorem~\ref{thm:structure}]
Let $r$ be chosen later (depending only on $\delta$) and define
\begin{equation}
\label{eq:C}
C_{r} = 1,
\qquad
C_{i-1} = \max \Big\{ C_{i}, \frac{2}{\eta(C_{i})} \Big\}.
\end{equation}
Let $V_{1},V_{2},V_{3}$ be the open unit balls of $\|\cdot\|$, $\|\cdot\|^{*}$, and $\|\cdot\|_{H}$, respectively.
Suppose that the conclusion fails, then for every $i \in \{1,\dots,r\}$ we have
\[
f \not\in C_{i}V_{1} + \eta(C_{i})V_{2} + \delta V_{3}.
\]
Since $V_{3}$ is open in $H$, Lemma~\ref{lem:sep} applies, and we obtain vectors $\phi_{i} \in H$ such that
\[
\<\phi_{i},f\> \geq 1,
\quad \|\phi_{i}\|^{*} \leq (C_{i})^{-1},
\quad \|\phi_{i}\|^{**} \leq \eta(C_{i})^{-1},
\quad \|\phi_{i}\| \leq \delta^{-1}.
\]
For every pair $i<j$ by \eqref{eq:C} we have
\[
|\<\phi_{i},\phi_{j}\>|
\leq \|\phi_{i}\|^{*} \|\phi_{j}\|^{**}\\
\leq (C_{i})^{-1} \eta(C_{j})^{-1}
\leq (C_{j-1})^{-1} \eta(C_{j})^{-1}
\leq \frac12,
\]
so that
\[
r^{2} \leq \<\phi_{1}+\dots+\phi_{r},f\>^{2}
\leq \|\phi_{1}+\dots+\phi_{r}\|^{2}
\leq r \delta^{-2} + \frac{r^{2}-r}{2},
\]
which is a contradiction if $r \geq 2 \delta^{-2}$.
\end{proof}

\section{The tree estimate}
\label{sec:tree}
For each $\scale\in\Z$ let $\DI_{\scale}$ be the collection of the dyadic cubes $I\subset\R^{\dm+1}$ of the form
\[
I=I_{0}\times\dots\times I_{\dm}=2^{\scale}(m_{0},\dots,m_{\dm}) + [0,2^{\scale}]^{\dm},
\qquad
\sum_{i}m_{i}=0.
\]
The \emph{scale} of a dyadic cube $I\in\DI_{\scale}$ is defined as $s(I):=\scale$.
Let also $\DI_{\Scales} := \cup_{\scale\in \Scales} \DI_{\scale}$ and $\DI:=\DI_{\Z}$.
This gives the splitting
\begin{equation}
\label{eq:tile-decomposition}
\Lambda_{\Scales}(F_{0},\dots,F_{\dm})
=
\sum_{I\in\DI_{\Scales}} \Lambda_{I}(F_{0},\dots,F_{\dm}),
\end{equation}
where for each $I\in\DI_{\scale}$ we have set
\[
\Lambda_{I}(F_{0},\dots,F_{\dm})
:=
\int_{\R^{\dm+1}} \prod_{i=0}^{\dm}F_{i}(x_{(i)}) \psi_{\scale}(\sum x) \prod_{i=0}^{\dm-1} 1_{I_{i}}(x_{i}) \dif x.
\]
Contrary to what could be expected, our argument would not benefit from using smoother versions of the cutoffs $1_{I_{i}}$.
However, this appears to be a limitation rather than a strength of our approach.

We write elements of $\R^{\dm+1}$ as $x=(x',x_{\dm}) \in \R^{\dm}\times\R$ and dyadic cubes $I\in\DI$ as $I'\times I_{\dm}$, where $I'$ is a dyadic cube in $\R^{\dm}$ and $I_{\dm}$ is a dyadic interval in $\R$.
A \emph{tree} with top $J\in\DI$ is a collection of boxes $I\in\DI$ such that $I'\subset J'$.
In this section we obtain a gain over the trivial bound (coming from Fubini's theorem) for the restriction of the sum \eqref{eq:tile-decomposition} to a tree.
\begin{proposition}
\label{prop:single-tree}
For every $\delta>0$ there exists $S_{\delta,\dm}\in\N$ such that
for any functions $F_{0},\dots,F_{\dm}:\R^{\dm}\to [0,1]$ and for every dyadic cube $J\in\DI$ there exists an interval of scales $\Scales'\subset\Z$ with $|\Scales'| \leq S_{\delta,\dm}$ and $\max \Scales'=s(J)$ such that
\[
\big| \sum_{\scale\in \Scales'} \sum_{I\in\DI_{\scale}: I'\subset J'} \Lambda_{I}(F_{0},\dots,F_{\dm}) \big|
\lesssim_{\dm}
\delta 2^{\dm s(J)} |\Scales'|.
\]
\end{proposition}
Note that $\Scales'$ depends both on the (bounded) functions $F_{i}$ and the dyadic square $J$, but $S_{\delta,\dm}$ does not.

\begin{proof}[Proof of Proposition~\ref{prop:single-tree}]
By scaling we may assume $s(J)=0$.
Note that
\begin{equation}
\label{eq:tree-single-scale}
\sum_{I\in\DI_{\scale} : I'\subset J'}\Lambda_{I}(F_{0},\dots,F_{\dm})
=
\int_{\R^{\dm+1}} \prod_{i=0}^{\dm} F_{i}(x_{(i)}) \psi_{\scale}(\sum x) \prod_{i=0}^{\dm-1} 1_{J_{i}}(x_{i}) \dif x.
\end{equation}
for every $\scale$ and the integrand is supported on $10J$, say.

A \emph{dual function} is a function from $X := \prod_{i=0}^{\dm-1} 10J_{i}$ to $\C$ of the form
\[
x\mapsto\prod_{A\subsetneq\{0,\dots,\dm-1\}}f_{A}(x|_{A}),
\]
where $f_{A} : \prod_{i\in A} 10J_{i}\to\C$ are functions bounded by $1$.
Denote the set of dual functions by $\Sigma$ and apply Theorem~\ref{thm:structure} with $H=L^{2}(X)$, $f=F_{\dm}|_{X}$, the extended seminorm given by Lemma~\ref{lem:Sigma-ext-seminorm} and a function $\eta$ to be chosen later.

To dispose of the $L^{2}$ error term note that at each scale $\scale\leq 0$ the right-hand side of \eqref{eq:tree-single-scale} is bounded by
\[
\int_{10 J} |F_{\dm}(x_{(\dm)})| |\psi_{\scale}(\sum x)| \dif\vec x\\
\lesssim
\|F_{\dm}\|_{L^{1}(10 J_{(\dm)})} \|\psi_{\scale}\|_{1}
\lesssim
\|F_{\dm}\|_{L^{2}(10 J_{(\dm)})}.
\]
The contribution of the uniform term (bounded in $\|\cdot\|_{\Sigma}^{*}$) is estimated by
\[
\sum_{\scale\in \Scales'} \|\hat\psi_{\scale}\|_{1} \Big| \int_{10 J} \prod_{i}F_{i}(x_{(i)}) e(\xi_{\scale}\sum x) \prod_{i=0}^{\dm-1} 1_{J_{i}}(x_{i}) \dif x \Big|
\]
for some choice of frequencies $\xi_{\scale}$.
Note that $\|\hat\psi_{\scale}\|_{1} \lesssim 2^{-\scale}$.
Inside the absolute value, the character splits into a product of functions depending on one variable each.
Since $\dm>1$, each function that depends on only one coordinate $x_{i}$ can be absorbed into one of the functions $F_{i}$, $i<\dm$.
Thus for each fixed $x_{\dm}$ the integral above is a pairing of $F_{\dm}$ with a dual function, and we obtain the estimate
\[
\sum_{\scale\in \Scales'} 2^{-\scale} \| F_{\dm} \|_{\Sigma}^{*}
\lesssim
2^{|\Scales'|} \|F_{\dm}\|_{\Sigma}^{*}.
\]

It remains to treat the structured term.
Suppose $F_{\dm}\in\Sigma$, so that
\[
F_{\dm} = \prod_{i=0}^{\dm-1} f_{i},
\]
where each function $f_{i}$ is bounded by $1$ and does not depend on the $i$-th coordinate.
Substituting this into \eqref{eq:tree-single-scale} we obtain
\[
\int_{10 J} \prod_{i=0}^{\dm-1} (F_{i}f_{i})(x_{(i)}) \psi_{\Scales'}(\sum x) \prod_{i=0}^{\dm-1}1_{J_{i}}(x_{i}) \dif x.
\]
This can be written as
\[
\int_{10 J_{\dm}}\int_{\R^{\dm}} \prod_{i=0}^{\dm-1} (1_{10 J}F_{i}1_{A_{i}} \prod_{j\neq i,\dm} 1_{J_{j}})(x'_{(i)},x_{\dm}) \psi_{\Scales'}(\sum x' + x_{\dm}) \dif x' \dif x_{\dm}.
\]
Changing variable in the inner integral and applying the inductive hypothesis (Theorem~\ref{thm:tiny-gain} with $\dm-1$ in place of $\dm$ and $p_{0}=\dots=p_{\dm-1}=\dm$) we can bound this by
\[
c_{\dm-1}(|\Scales'|) |\Scales'|
\]
with a monotonically decreasing function $c_{\dm-1}$ such that $\lim_{|\Scales'|\to\infty}c_{\dm-1}(|\Scales'|) = 0$.
Summing the contributions of the three terms given by Theorem~\ref{thm:structure} we obtain
\[
\delta |\Scales'| + 2^{|\Scales'|} \eta(C) + C c_{\dm-1}(|\Scales'|) |\Scales'|,
\]
where $C=O_{\delta,\eta}(1)$.
Choose a monotonically increasing function $\tilde S_{\delta,\dm} : \R_{+}\to\N$ such that $a c_{\dm-1}(\tilde S_{\delta,\dm}(a)) \leq \delta$ for all $a$.
Let $\eta(a) := \delta \tilde S_{\delta,\dm}(a)2^{-\tilde S_{\delta,\dm}(a)}$.
Then the claim follows with $|\Scales'|=\tilde S_{\delta,\dm}(C)$.
\end{proof}

\begin{corollary}
\label{cor:single-tree}
Let $\delta>0$ and $S_{\delta,\dm}$ be the number from Proposition~\ref{prop:single-tree}.
Then for every $J\in\DI$ and every interval $\Scales'\subset\Z$ with $\max \Scales'=s(J)$, we have
\[
\big| \sum_{\scale\in \Scales'} \sum_{I\in\DI_{\scale} : I'\subset J'} \Lambda_{I}(F_{0},\dots,F_{\dm}) \big|
\lesssim_{\dm}
2^{\dm s(J)} (\min(|\Scales'|, S_{\delta,\dm}) + \delta \max(|\Scales'|-S_{\delta,\dm},0))
\]
for any functions $F_{0},\dots,F_{\dm}$ bounded by $1$.
\end{corollary}
\begin{proof}
By induction on $|\Scales'|$.
For $|\Scales'|\leq S_{\delta,\dm}$ the estimate follows from $|F_{i}|\leq 1$ and $\|\psi_{\scale}\|_{1}\lesssim 1$.

If $|\Scales'|> S_{\delta,\dm}$, then by Proposition~\ref{prop:single-tree} we can find a final interval $\Scales'' \subset \Scales'$ such that the sum over $\Scales''$ can be estimated by $2^{\dm s(J)}\delta |\Scales''|$.
The remaining part of the sum splits into sums over subintervals of scale $s(J)-|\Scales''|$, and to those we apply the Corollary with $\Scales'\setminus \Scales''$ in place of $\Scales'$.
\end{proof}

\section{Tree selection}
\label{sec:selection}
For cubes $I\in\DI_{\scale}$ write
\[
a_{I} := 2^{-\scale \dm}\Lambda_{I}(F_{0},\dots,F_{\dm})
\]
The integrand in the definition of $\Lambda_{I}$ vanishes outside $10 I$, say, and by the Loomis--Whitney inequality we can estimate
\[
|a_{I}| \lesssim \prod_{i=0}^{\dm} \min_{\pi_{\Delta} I} M_{\dm}F_{i},
\]
where $\pi_{\Delta} I$ is the subset of the diagonal $\Delta = \{ x\in\R^{\dm+1} : \sum x = 0\}$ consisting of the points whose first $\dm$ coordinates lie in $I'$ and $M_{\dm}$ if the maximal function $M_{\dm}F(x) = \sup_{Q\ni x} (|Q|^{-1}\int_{Q} |F|^{\dm})^{1/\dm}$.
Raising this to a power $\alpha$ and summing over the squares $I$ of a given size we obtain
\[
\sum_{I\in\DI_{\scale}} |a_{I}|^{\alpha} 2^{\scale \dm}
\lesssim
\int_{\Delta} \prod_{i=0}^{\dm} M_{\dm}F_{i}(x_{(i)})^{\alpha}.
\]
By Hölder's inequality this is bounded by
\[
\prod_{i=0}^{\dm} \| (M_{\dm}F_{i})^{\alpha} \|_{p_{i}}
=
\prod_{i=0}^{\dm} \| M_{\dm}F_{i} \|_{\alpha p_{i}}^{\alpha}
\lesssim
\prod_{i} |E_{i}|^{1/p_{i}}
\]
provided $\alpha p_{i}>\dm$ for all $i$.
It follows from $p_{i}>\dm$ that 
\begin{equation}
\label{eq:small-tiles}
\sum_{\scale\in \Scales} \sum_{I : |a_{I}| < \delta, l(I)=\scale} |a_{I}| 2^{\dm\scale}
\leq
\sum_{\scale\in \Scales} \sum_{I : l(I)=2^{\scale}} |a_{I}|^{\alpha} \delta^{1-\alpha} 2^{\dm\scale}
\lesssim_{\alpha}
\delta^{1-\alpha} |\Scales| \prod_{i} |E_{i}|^{1/p_{i}}
\end{equation}
for every $\alpha$ sufficiently close to $1$ and every $\delta>0$.

Let now $\mathcal{J}\subset\DI$ be the collection of maximal cubes $J$ with $|a_{J}|>\delta$.
The union of these cubes cannot be too large.
Indeed, we have
\[
\pi_{\Delta} \bigcup\{J : |a_{J}|>\delta\}
\subset
\Delta \cap \{ \prod_{i=0}^{\dm} M_{\dm}1_{E_{i}} \gtrsim \delta \}.
\]
The measure of the latter set is bounded by
\[
\delta^{-1} \|\prod_{i=0}^{\dm} M_{\dm}1_{E_{i}}\|_{L^{1}(\Delta)}
\leq
\delta^{-1} \prod_{i=0}^{\dm} \| M_{\dm}1_{E_{i}}\|_{L^{p_{i}}(\Delta)}
\lesssim
\delta^{-1} \prod_{i=0}^{\dm} | E_{i} |^{1/p_{i}},
\]
where we have again used $p_{i}>\dm$.
Let $S = S_{\delta^{2},\dm}$ be the number given by Proposition~\ref{prop:single-tree} with $\delta^{2}$ in place of $\delta$.
For those $J\in\mathcal{J}$ with $s(J)> \min \Scales + \delta^{-2} S$ by Corollary~\ref{cor:single-tree} we have
\[
\big| \sum_{I\in\DI_{\Scales} : I'\subset J'} \Lambda_{I}(F_{0},\dots,F_{\dm}) \big|
\lesssim_{\dm}
2^{\dm s(J)} \delta^{2} |\Scales|.
\]
In particular,
\[
\sum_{J\in\mathcal{J} : s(J)> \min \Scales + \delta^{-2} S}
\big| \sum_{I\in\DI_{\Scales} : I'\subset J'} \Lambda_{I}(F_{0},\dots,F_{\dm}) \big|
\lesssim_{\dm}
\delta^{2} |\Scales| \Big| \pi_{\Delta} \bigcup\{J : |a_{J}|>\delta\} \Big|
\lesssim
\delta |\Scales| \prod_{i} |E_{i}|^{1/p_{i}}.
\]
On the other hand, by \eqref{eq:small-tiles} with $\alpha=0$ we have
\[
\sum_{J\in\mathcal{J} : s(J) \leq \min \Scales + \delta^{-2} S} \sum_{I\in\DI_{\Scales} : I'\subset J'} |\Lambda_{I}(F_{0},\dots,F_{\dm})|
\lesssim
\delta^{-2} S \prod_{i} |E_{i}|^{1/p_{i}}.
\]
Summing the above contributions we obtain the claim of Theorem~\ref{thm:tiny-gain} (in the case of characteristic functions).

\appendix
\section{Maximally modulated multilinear Hilbert transform}
\label{sec:max-mod}
In this appendix we show how to encode some one-dimensional multilinear operators in the simplex Hilbert transform using the ideas from \cite[Appendix B]{arxiv:1506.00861}.
Consider the family of multilinear forms
\begin{equation}
\label{eq:HHT-def-beta}
\Lambda_{\vec\beta_{0},\dots,\vec\beta_{\dm}}(F_{0},\dots,F_{\dm})
:=
\int_{\R^{\dm}} \int_{\R} \prod_{i=0}^{\dm}F_i(\vec x-\vec\beta_{i}t) K(t)\dif t \dif \vec x,
\end{equation}
where $\vec\beta_{i}\in\R^{\dm}$ are in general position.
The main observation is that
\begin{equation}
\label{eq:HHT-beta-hom}
\|\Lambda_{K}\|_{L^{p_{0}} \times \dots \times L^{p_{\dm}}}
=
|\det B|^{1/p_{0}+\dots+1/p_{\dm}-1}
\|\Lambda_{\vec\beta_{0},\dots,\vec\beta_{\dm}}\|_{L^{p_{0}} \times \dots \times L^{p_{\dm}}},
\end{equation}
where
\[
B :=
\begin{pmatrix}
1 & \dots & 1\\
\vec\beta_{0} & \dots & \vec\beta_{\dm}
\end{pmatrix}.
\]
In particular, the norm of \eqref{eq:HHT-def-beta} does not depend on the $\vec\beta_{i}$'s provided $\sum_{i=0}^{\dm}1/p_{i}=1$.

\begin{proof}[Proof of \eqref{eq:HHT-beta-hom}]
Consider the change of variables
\[
\begin{pmatrix}
t\\ \vec u
\end{pmatrix}
=
B \vec x,
\quad
\quad
\vec u\in\R^{\dm},
\vec x\in\R^{\dm+1}.
\]
If $\pi_i:\R^{\dm+1}\to\R^{\dm}$ denotes omission of the $i$-th coordinate, then for arbitrary functions $F_0,\dots,F_\dm$ we have
\begin{align*}
\Lambda(F_{0},\dots,F_{\dm})
&=
\int_{\R^{\dm+1}} \prod_{i=0}^{\dm}F_{i}(\pi_i(\vec x)) K(\sum \vec x) \dif \vec x\\
&=
|\det B|^{-1}
\int_{\R\times\R^{\dm}}
\prod_{i=0}^{\dm} F_{i}(\pi_i B^{-1}(t,\vec u))
K(t) \dif t \dif \vec u\\
&=
|\det B|^{-1}
\int_{\R\times\R^{\dm}}
\prod_{i=0}^{\dm} \tilde F_{i}(\vec u-\vec{\beta_{i}}t)
K(t) \dif t \dif \vec u\\
&=
|\det B|^{-1} \Lambda_{\vec\beta_{0},\dots,\vec\beta_{\dm}}(\tilde F_{0}, \dots, \tilde F_{\dm}),
\end{align*}
where
\begin{equation}\label{eq:fn-substitution}
\tilde F_{i}(\vec u) := F_{i}(\pi_i B^{-1}(0,\vec u)).
\end{equation}
Here we have used the fact that
\[
\pi_i B^{-1}
\begin{pmatrix}
1 \\ \vec\beta_{i}
\end{pmatrix}
=
0.
\]
The important observation is now that
\[
\|\tilde F_{i}\|_{p_{i}} = |\det B|^{1/p_{i}} \|F_{i}\|_{p_{i}}.
\]
Indeed, the change of variables in the definition of $\tilde F_{i}$ is given by the $\dm\times \dm$ submatrix of $B^{-1}$ obtained by crossing out the first column and the $i$-th row.
By Cramer's rule the determinant of that submatrix equals $\det B^{-1}$ times the $(1,i)$-th entry of $B$, up to the sign.
Since the latter entry of $B$ is $1$, the determinant of the change of variables is $\pm(\det B)^{-1}$.
The ratio of $L^{p_{i}}$ norms equals the absolute value of the determinant to the power $-1/p_{i}$, as required.
The claim now follows after taking a supremum over the $F_{i}$'s.
\end{proof}

Now we encode the maximally modulated multilinear Hilbert transform \eqref{eq:max-mod-MHT} in \eqref{eq:HHT-def-beta}.
Let $\vec\beta_{0}\in\R^{\dm}$ be the origin, $\vec\beta_{j}=e_{j}$ for $j=1,\dots,d$, $\vec\beta_{d+1}=b_{d+1}e_{d+1}$, and $\vec\beta_{j}=b_{j}e_{d+1}+\epsilon e_{j}$ for $j=d+2,\dots,\dm$.
The identity
\[
\sum_{j=0}^{k} (-1)^{k-j} \binom{k}{j} j^m
=
\begin{cases}
0 & \text{ for } m=0,1,\ldots,k-1,\\
k! & \text{ for } m=k
\end{cases}
\]
can be shown by induction on a positive integer $k$ and its immediate consequence is
\[
\sum_{j=0}^{d} \sum_{k=\max\{j,1\}}^{d} (-1)^{j}\frac{1}{k!}\binom{k}{j}N_k(x_{d+1})\Big(\sum_{l=1}^{k}l x_l-jt\Big)^k
=
\sum_{k=1}^{d} N_k(x_{d+1}) t^k .
\]
It follows that with
\[
F_j(x_1,\ldots,x_\dm)
=
g_j(x_{d+1}) \prod_{k=\max\{j,1\}}^{d} \exp\bigg(i(-1)^{j}\frac{1}{k!}\binom{k}{j}N_k(x_{d+1})\Big(\sum_{l=1}^{k}l x_l\Big)^k\bigg)
\]
for $j=0,1,\ldots,d$ and
\[
F_{j}(x_1,\ldots,x_\dm) = f_{j}(x_{d+1})
\]
for $j=d+1,\dots,\dm$ the form \eqref{eq:HHT-def-beta} formally becomes
\[
\int g_0 g_1\cdots g_{d} C_{N_1,\ldots,N_{d}}(f_{d+1},\dots,f_{\dm}).
\]
To be precise we should multiply each $F_{j}$ by $\prod_{j\neq d+1}\phi(x_{j})$, where $\phi$ is a fixed smooth non-negative cutoff function, and then let $\epsilon\to 0$.

\begin{ack}
I thank Terence Tao and Vjekoslav Kova\v{c} for useful discussions about the triangular Hilbert transform and Christoph Thiele for persistent questions that helped to streamline the argument.
\end{ack}

\printbibliography
\end{document}
